\documentclass[a4,12pt]{amsart}
\usepackage{amssymb,amstext,amsmath,amscd,amsthm,amsfonts,enumerate,graphicx,latexsym,color,stmaryrd}
\usepackage[all]{xy}

\newtheorem{thm}{Theorem}[section]
\newtheorem{cor}[thm]{Corollary}
\newtheorem{lem}[thm]{Lemma}
\newtheorem{prop}[thm]{Proposition}

\newtheorem*{thm*}{Theorem}

\theoremstyle{definition}
\newtheorem{dfn}[thm]{Definition}
\newtheorem{rem}[thm]{Remark}

\newtheorem{ex}[thm]{Example}

\newtheorem*{claim*}{Claim}
\newtheorem*{ques*}{Question}

\newtheorem*{ac}{Acknowledgments}
\theoremstyle{remark}
\numberwithin{equation}{thm}
\def\m{\mathfrak{m}}
\def\n{\mathfrak{n}}
\def\q{\mathfrak{q}}
\def\s{\operatorname{s}}

\def\Ker{\operatorname{Ker}}

\def\Coker{\operatorname{Coker}}
\def\Gdim{\operatorname{G-dim}}

\def\H{\operatorname{H}}
\def\Tor{\operatorname{Tor}}
\def\Ext{\operatorname{Ext}}
\def\cx{\operatorname{cx}}
\def\gr{\operatorname{grade}}

\def\e{\operatorname{e}}
\def\height{\operatorname{ht}}

\def\Hom{\operatorname{Hom}}
\def\Ext{\operatorname{Ext}}

\def\curv{\operatorname{curv}}

\def\V{\operatorname{V}}
\def\supp{\operatorname{Supp}}
\def\Z{\mathbb{Z}}
\def\id{\mathrm{id}}

\def\cc{\mathfrak{c}}
\def\p{\mathfrak{p}}

\begin{document}
\setlength{\baselineskip}{15pt}
\title{Totally reflexive modules and Poincar\'e series}
\author{Mohsen Gheibi}
\address{Mohsen Gheibi, Department of Mathematics, University of Texas--Arlington,
411 S. Nedderman Drive,
Pickard Hall 445,
Arlington, TX, 76019, USA }
\email{mohsen.gheibi@uta.edu}
\author{Ryo Takahashi}
\address{Ryo Takahashi, Graduate School of Mathematics, Nagoya University, Furocho, Chikusaku, Nagoya{, Aichi} 464-8602, Japan{/Department of Mathematics, University of Kansas, Lawrence, KS 66045-7523, USA}}
\email{takahashi@math.nagoya-u.ac.jp}
\urladdr{http://www.math.nagoya-u.ac.jp/~takahashi/}
\thanks{Takahashi was partly supported by JSPS Grant-in-Aid for Scientific Research 16K05098 and JSPS Fund for the Promotion of Joint International Research 16KK0099}

\begin{abstract}
We study Cohen-Macaulay non-Gorenstein local rings $(R,\m,k)$ admitting certain totally reflexive modules. {More precisely, we}
give a description of the Poincar\'{e} series of $k$ by using the Poincar\'{e} series of  {a non-zero} totally reflexive module {with minimal multiplicity}.
Our results generalize a result of Yoshino to higher-dimensional Cohen-Macaulay local rings.
Moreover, from a quasi-Gorenstein ideal satisfying some conditions, we construct a family of non-isomorphic indecomposable totally reflexive modules having an arbitrarily large minimal number of generators.
\end{abstract}

\maketitle
\section{Introduction}

In the 1960s, Auslander and Bridger \cite{A,ABr} introduced the notion of Gorenstein dimension, G-dimension for short, for finitely generated modules as a generalization of projective dimension.
Over Gorenstein rings, every finitely generated module has finite G-dimension, similarly to the fact that over regular rings every module has finite projective dimension, and the class of maximal Cohen-Macaulay modules coincides with the class of modules with G-dimension zero, which are called totally reflexive modules.
Over rings that are not G-regular in the sense of \cite{greg} {(see Definition \ref{gregular})}, there exist modules of finite G-dimension and infinite projective dimension, which is equivalent to the existence of non-free totally reflexive modules.
Yoshino \cite[Theorem 3.1]{Y} proved the following remarkable result for Artinian short local rings {possessing non-free totally reflexive modules}.

\begin{thm*}[Yoshino]
Let $(R,\m,k)$ be {a} non-Gorenstein local ring of type $r$ such that $\m^3=0\ne\m^2$.
Assume {that} $R$ contains $k$, and that there exists a non-free totally reflexive $R$-module $M$.
Then the following statements hold.
\begin{enumerate}[\rm(1)]
\item
The ring $R$ {is}  a Koszul algebra.
One has the Poincar\'{e} series $P^R_k(t)=\frac{1}{(1-t)(1-rt)}$, the Bass series $I^R_R(t)=\frac{r-t}{1-rt}${,}  and the Hilbert series $H_R(t)=(1+t)(1+rt)$.
\item The Poincar\'{e} series of $M$ is $P^R_M(t)=\frac{\nu_R(M)}{1-t}$.
\end{enumerate}
\end{thm*}

\noindent
{Here, $\nu_R(M)$ stands for the minimal number of generators of $M$.}

Later, Christensen and Veliche \cite{CV} generalized Yoshino's theorem in terms of acyclic complexes over a non-Gorenstein local ring $(R,\m)$ with $\m^3=0\ne\m^2$; see \cite[Theorem A]{CV}.

In celebration of Yoshino's theorem, Golod and Pogudin \cite{GP} call a ring satisfying the assumption of the theorem, a \emph{Yoshino algebra}.
It turns out that over Yoshino algebras, every non-free indecomposable totally reflexive module $M$ satisfies $\m^2M=0$.
Gerko \cite{Ger} showed that an Artinian non-Gorenstein local ring $(R,\m)$ of type $r$ admitting a non-free totally reflexive module $M$ with $\m^2M=0$ has Bass series $I^R_R(t)=\frac{r-t}{1-rt}$. This naturally leads us to the following question.

\begin{ques*}
Let $(R,\m,k)$ be an Artinian non-Gorenstein local ring admitting a  {non-zero} totally reflexive module $M$ with $\m^2M=0$.
Then what can one say about the Poincar\'{e} series of $k$?
\end{ques*}

In this paper we answer the above question, including even the higher-dimensional case.
More explicitly, our main result is the following. We denote by $(-)^*$ the algebraic dual functor $\Hom_R(-,R)$.

\begin{thm}\label{Main1}
Let $(R,\m,k)$ be a $d$-dimensional Cohen-Macaulay non-Gorenstein local ring of type $r$.
Let $M\ne0$ be a totally reflexive $R$-module with $\m^2M\subseteq \q M$ for some parameter ideal $\q$ of $R$.
Then the following hold.
\begin{enumerate}[\rm(1)]
\item
One has
$$
P^R_k(t)=\frac{P^R_M(t)(1+t)^d}{\nu_R(M)(1-rt)}.
$$ {In particular, if $M'\neq 0$ is another totally reflexive $R$-module with $\m^2M' \subseteq \q M'$, then
$$
\displaystyle\frac{P^R_M(t)}{\nu_R(M)}=\frac{P^R_{M'}(t)}{\nu_R(M')}.
$$ }
\item
For each {totally reflexive} $R$-module $N$ there is an inequality of power series
$$
\frac{P^R_N(t)}{1-rt}\preccurlyeq \frac{C\cdot P^R_M(t)}{1-rt},\text{ where }C=\frac{r\cdot\ell_R(N/\q N)-\ell_R(N^*/\q N^*)}{(r^2-1)\nu_R(M)}.
$$
\end{enumerate}
\end{thm}
\noindent{Here $\ell_R(X)$ stands for the length of an $R$-module $X$.}

Levin \cite{Lev} introduced the notion of a large homomorphism {of local rings}: a surjective homomorphism $f:R\rightarrow S$ of local rings is called \emph{large} if the induced map $f_*:\Tor^R_{*}(k,k)\rightarrow \Tor^S_{*}(k,k)$ of graded algebras is surjective. It has been shown that $f:R\rightarrow S$ is large if and only if $P^R_M(t)=P^S_M(t)P^R_S(t)$ for every finitely generated $S$-module $M$. {There are} few examples of {large homomorphisms} in the literature; see \cite[Theorem 2.2 and 2.4]{Lev}. {Also, in} \cite[Theorem 6.2]{AHS} {it is shown that the natural homomorphism $R\rightarrow R/I$, where $I$ is a quasi-complete intersection ideal with $I\cap\m^2=I\m$, is large}. As an application of  Theorem \ref{Main1}, we {give} another {example} of large homomorphisms.

\begin{thm}\label{C21}
Let $(R,\m)$ be a Cohen-Macaulay non-Gorenstein local ring.
Let $I$ be a proper ideal of $R$ with $\m^2\subseteq I$ and $\Gdim_R(R/I)<\infty$.
Then
\begin{enumerate}[\rm(1)]
\item
The homomorphism $R\rightarrow R/I$ is large.
\item
If $R$ is a standard graded ring, then $R$ is a Koszul ring if and only if $R/I$ is a Koszul $R$-module.
\end{enumerate}
\end{thm}

There are ideals satisfying in the hypothesis of Theorem \ref{C21} which are not quasi-complete intersection; see Example \ref{EX}(2). Moreover, we show that Theorem \ref{C21} fails if $\Gdim_R(R/I)=\infty$; see Example \ref{short}.

Our second aim in this paper is to give a construction of infinitely many non-isomorphic indecomposable totally reflexive modules from given ideals of arbitrary local rings. Construction of indecomposable totally reflexive modules with specific properties is widely investigated recently; see \cite{AV,CJRSW,Holm,CGT,RSW,GP}.
Here, all the references except \cite{GP} use exact pairs of zero-divisors for their construction.
There exist non-Gorenstein rings {that do not admit an exact pair of zero-divisors but admit a non-free totally reflexive module}; see for example \cite[Propositions 9.1 and 9.2]{CJRSW}.
Our construction uses quasi-Gorenstein ideals, which are the generalization of ideals generated by exact zero-divisors, and recovers the results of \cite{CGT,RSW}; see Examples \ref{qci2} and \ref{EX}. More precisely, we prove the following in section 4.

\begin{thm}\label{2219}
Let $(R,\m,k)$ be a local ring.
Let $I$ be  {a quasi-Gorenstein} ideal of $R$ with $\gr I=0$ and $\dim_RI/I^2\ge2$.
Suppose that either of the following holds.
\begin{enumerate}[\rm(a)]
\item
 {$(0:_RI^2)=(0:_RI)$, or}
\item
 $R/I$ is G-regular.
\end{enumerate}
Then for every integer $n\ge1$, there exists a family of exact sequences of $R$-modules
$$
\{0 \to (R/I)^n \to M_{{n,}u} \to R/(0:_RI) \to 0\}_{u\in k}
$$
such that the modules $M_{{n,}u}$ are non-free and pairwise {non-isomorphic} indecomposable {totally reflexive} $R$-modules with $\nu_R(M_{n,u})=n+1$.
In particular, if $k$ is infinite, then  {for each $n\geq 1$} there exist infinitely many such modules $M_{{n,}u}$.
\end{thm}

Finally, we give examples of our construction over rings that satisfy {the} conditions of Theorem \ref{Main1}; see Examples \ref{qci2} and  \ref{EX}(3).

\section{Preliminaries}

Throughout $R$ is a commutative Noetherian local ring with the maximal ideal $\m$.
For an $R$-module $M$ denote $\beta^R_n(M)$ ($\mu^n_R(M)$) the $n$-th \emph{Betti number} ($n$-th \emph{Bass number}) of $M$.  The \emph{Poincar\'{e} series} (\emph{Bass series}) of $M$ over $R$ is defined by
$$P^R_M(t)=\sum^{\infty}_{i=0}\beta^R_n(M)t^n \ \ (I^M_R(t)=\sum^{\infty}_{i=0}\mu^n_R(M)t^n).$$
The \emph{complexity} of $M$ over $R$ is defined by
$$\cx_R{(}M{)} = \inf \{d \in \mathbb{N} \mid \exists \ r \in \mathbb{R} \ such \ that \ \beta^R_n(M) \leq rn^{d-1} \ for \ n \gg 0 \}.$$
{The \emph{curvature} of $M$ over $R$ is defined by $\curv_R(M)=\underset{n\rightarrow \infty}\limsup \sqrt[n]{\beta^R_n(M)}$.}

\begin{dfn}\label{gdzero}
A finitely generated $R$-module $M$ is called \emph{totally reflexive} (or of $G$-\emph{dimension zero}) if the evaluation homomorphism $M\to M^{**}$ is an isomorphism and $\Ext_R^i(M,R)=0=\Ext_R^i(M^*,R)$ for all $i>0$.
\end{dfn}

Recall that the infimum  {non-negative} integer $n$ for which there exists an exact sequence $$0 \to G_{n} \to \dots  \to G_{0} \to M\to 0,$$
such that each $G_{i}$ is {a} totally reflexive {$R$-module}, is called the \emph{Gorenstein dimension} of $M$. If $M$ has Gorenstein dimension $n$, we write $\Gdim_R(M)=n$.

\begin{dfn}\label{df1} A proper ideal $I$ of $R$ is called \emph{quasi-Gorenstein} in sense of \cite{AF} if $$\Ext^i_R(R/I,R)\cong \left\lbrace
           \begin{array}{c l}
              R/I\ \ & \text{ \ \ $i=\gr I$,}\\
              0\ \   & \text{   \ \ $\textrm{otherwise}$.}
           \end{array}
        \right.$$
\end{dfn}

\begin{rem}\label{remqg}  Let $I$ be a quasi-Gorenstein ideal. Then the following hold.
\begin{enumerate}[(1)]
\item
The ideal $I$ is a $G$-perfect ideal in the sense of Golod \cite{Gol}, that is, $\Gdim_R(R/I)=\gr I$. In particular, $\Omega^{\gr I}(R/I)$ is a
totally reflexive $R$-module.
\item
Suppose that $\gr I=0$, i.e., $(0:_RI)\ne0$. One then has {
\par
 $(0:_RI)\cong (R/I)^* \cong R/I$,\qquad
 $R/(0:_RI)\cong I^*$, and \qquad
 $(0:_R(0:_RI))=I$. } \\
In particular, {$(0:_RI)$ is a principal ideal, and} $R/I$ {and $R/(0:_RI)$ are} totally reflexive $R$-modules.
\end{enumerate}
\end{rem}

We recall the definition of a quasi-complete intersection ideal from \cite{AHS}.

\begin{dfn} Let $I$ be an ideal of a local ring $R$. Set $\overline{R}=R/I$ and let $K$ be the Koszul complex with respect to a minimal generating set of $I$. Then $I$ is said to be \emph{quasi-complete intersection} if $\H_1(K)$ is free $\overline{R}$-module, and the natural homomorphism
$$\lambda_*^{\overline R}:\wedge^{\overline{R}}_*\H_1(K)\longrightarrow \H_*(K)$$
of graded $\overline{R}$-algebras with $\lambda^{\overline{R}}_1=\id_{\H_1(K)}$, is bijective.
\end{dfn}

\begin{rem}\label{R11} Every quasi-complete intersection ideal is quasi-Gorenstein; see \cite[Theorem 2.5]{AHS}.
\end{rem}

\begin{lem}\label{lem11} Let $I$ be a quasi-Gorenstein ideal of $R$. Then for any finitely generated $R/I$-module $M$, one has $\Gdim_R(M)=\Gdim_{R/I}{(}M{)} +\gr I$. In particular, if $\gr I=0$, then $M$ is a totally reflexive $R$-module if and only if $M$ is a totally reflexive $R/I$-module.
\end{lem}
\begin{proof} {This follows from \cite[Theorem 5]{Gol} by setting $K=R$.}
\end{proof}

\begin{dfn}\label{gregular}
A ring $R$ is called {\em G-regular} if all totally reflexive $R$-modules are free.
\end{dfn}

Every regular ring is G-regular. More generally, a Cohen-Macaulay local ring with minimal multiplicity is G-regular; see Corollary \ref{Greg}.
The following proposition{,}  stated in \cite[Proposition 4.6]{greg},  provides examples of  G-regular local ring{s}. Other examples of G-regular local rings are found in \cite[Example 3.5]{AM}, \cite[Examples 5.2--5.5]{greg} and \cite[Corollary 4.8]{NS}; see also \cite[Corollary 6.6]{fiber}.

\begin{prop}\label{46}
Let $(R,\m)$ be a G-regular local ring and $x\in\m$ an $R$-regular element.
Then $R/(x)$ is G-regular if and only if $x\notin\m^2$.
\end{prop}

\begin{dfn}  Let $R=\oplus^\infty_{i=0}R_i$ be a graded ring with $R_0=k$ a field. A graded $R$-module $M$ is called \emph{Koszul} or \emph{linear}, if $\beta^R_{i,j}(M)=0$ for all $j-i\neq r$ and some $r\in \mathbb{Z}$. More precisely, Koszul modules are those modules that have linear resolutions. The ring $R$ is said to be \emph{Koszul} if $k$ is a Koszul $R$-module. For more details, see {\cite{Frob,HI}}.
\end{dfn}

\section{Totally reflexive modules with {minimal multiplicity}}

{In this section we give the proof of our main result, Theorem \ref{Main1}. Before, we bring a lemma.}

{
\begin{lem}\label{type} Let $(R,\m,k)$ be an Artinian local ring of type $r\geq 2$. Suppose there exists a non-zero totally reflexive $R$-module $M$ such that $\m^2 M=0$. Then the following statements hold.
{\begin{enumerate}[\rm(1)]
\item There exists an exact sequence $0\rightarrow k^{br}\rightarrow M \rightarrow k^b\rightarrow 0$ where $b=\nu_R(M)$.
\item One has $I^R_R(t)=\frac{r-t}{1-rt}$.
\end{enumerate}}
\end{lem}
\begin{proof} {It follows from \cite[Propositions 5.1 and 5.2]{Ger} by setting $K=R$.}
\end{proof}}

\begin{proof}[Proof of Theorem \ref{Main1}]
(1) First we show that $I^R_{R}(t)=\frac{t^d(r-t)}{1-rt}$. {Lemma \ref{type}} settles the case $d=0$.
In the case where $d{>0}$,{ $\q$ is generated by a regular sequence of length $d$. Therefore} there is an isomorphism {$\Ext^i_{R/\q}(k,R/\q)\cong \Ext^{i+d}_R(k,R)$}, which gives $\mu^i(R/\q)=\mu^{i+d}(R)$ for $i\in\Z$. Hence $$I^R_{R}(t)=t^dI^{R/\q}_{{R/\q}}(t)=t^d\cdot\frac{r-t}{1-rt}.$$

Set{$\overline{(-)}=(-)\otimes_RR/\q$ and let} $\nu_R(M)={b}$.
Replacing $R$ with its completion, we may assume that $R$ admits the canonical module $\omega$.
{As $\overline{M}$ is a totally reflexive $\overline{R}$-module,} $\Gdim_R({\overline{M}}) < \infty$ {by Lemma \ref{lem11}}. Hence we get $\Tor^R_{>0}({\overline{M}},\omega)=0$ by \cite[Theorems (3.1.10) and (3.4.6)]{C}.
Consider the exact sequence $0\rightarrow \Omega \omega \rightarrow F \rightarrow \omega\rightarrow 0$ where $F$ is a free $R$-module and $\Omega \omega \subseteq \m F$. As $\Tor^R_1({\overline{M}},\omega)=0$, applying ${\overline{M}}\otimes_R-$, we get an injection ${\overline{M}}\otimes_R\Omega \omega \hookrightarrow {\overline{M}}\otimes_R F$ whose image is a subset of $\m({\overline{M}}\otimes_RF)$. Since $\m^2({\overline{M}})=0$, it follows that ${\overline{M}}\otimes_R\Omega \omega$ is a $k$-vector space, that is, ${\overline{M}}\otimes_R\Omega \omega\cong k^{{b}s}$, where $s=\nu_R(\Omega\omega)$. Putting this together with $\Tor^R_{>0}({\overline{M}},\Omega\omega)=0$, one gets
$$
{b}s\,P^R_k(t)=P^R_{{\overline{M}}\otimes_R\Omega \omega}(t)=P^R_{{\overline{M}}}(t)\,P^R_{\Omega \omega}(t);
$$
see \cite[Theorem (A.7.6)]{C}.
Since $M$ is maximal Cohen-Macaulay {R-module,}
$$
P^R_{\overline{M}}(t)=P^R_M(t)(1+t)^d.
$$
{We have
$$
P^{{R}}_{\Omega\omega}(t)=\frac{1}{t}(P^{{R}}_\omega(t)-r)=\frac{1}{t}\left(\frac{1}{t^d}I^R_{{R}}(t)-r \right)=\frac{r^2-1}{1-rt},
$$
which especially says $s=r^2-1$.}
The result now follows.

{If $M'\neq 0$ is another totally reflexive $R$-module such that $\m^2 \subseteq \q M'$ then by the last argument $\displaystyle\frac{P^R_M(t)(1+t)^d}{\nu_R(M)(1-rt)} =P^R_k(t) = \frac{ P^R_{M'}(t)(1+t)^d}{\nu_R(M')(1-rt)}$. Thus $\displaystyle\frac{P^R_M(t)}{\nu_R(M)}=\frac{P^R_{M'}(t)}{\nu_R(M')}$.}

(2) {Note that} $\overline M$ and $\overline N$ are totally reflexive $\overline R$-modules with $P^{\overline{R}}_{\overline{M}}(t)=P^R_M(t)$ and $P^{\overline{R}}_{\overline{N}}(t)=P^R_N(t)$.
It follows from \cite[Proposition 3.3.3(a)]{BH} that $\overline{N^*}\cong\Hom_{\overline R}(\overline N,\overline R)$.
Thus, replacing $R$ with $R/\q$, we may assume $d=0$.
The same argument as in part (1) shows that
$$
P^R_L(t)= \frac{r^2-1}{1-rt}P^R_N(t),
$$
where $L:=N\otimes_R\Omega\omega$.
It is easy to observe by induction on $\ell_R(L)$ that $P^R_L(t)\preccurlyeq \ell_R(L)\cdot P^R_k(t)$.
Using (1), we obtain
$$
\frac{P^R_N(t)}{1-rt} \preccurlyeq \frac{C\cdot P^R_M(t)}{1-rt},\text{ where }C=\frac{\ell_R(L)}{(r^2-1)\nu_R(M)}.
$$
There is an exact sequence $0\to\Omega\omega\to R^{\oplus r}\to\omega\to0$.
Taking the Matlis dual $(-)^\vee=\Hom_R(-,\omega)$, we get an exact sequence $0\to R\to\omega^{\oplus r}\to(\Omega\omega)^\vee\to0$.
Applying the functor $\Hom_R(N,-)$ to this and noting that $N$ is totally reflexive, we obtain an exact sequence $0\to N^*\to(N^\vee)^{\oplus r}\to\Hom_R(N,(\Omega\omega)^\vee)\to0$.
It follows that
$$
\ell_R(L)=\ell_R(L^\vee)=\ell_R(\Hom_R(N,(\Omega\omega)^\vee))=\ell_R((N^\vee)^{\oplus r})-\ell_R(N^*)=r\cdot\ell_R(N)-\ell_R(N^*),
$$
which yields the equality
$$
C=\frac{r\cdot\ell_R(N)-\ell_R(N^*)}{(r^2-1)\nu_R(M)}.
$$
Thus (2) follows.
\end{proof}

\begin{cor}\label{fincxr}
Keep the notation of Theorem \ref{Main1}.
\begin{enumerate}[\rm(1)]
\item
If $M$ has finite complexity, then there exists a real number $U>1$ such that
$$
P^R_k(t)\preccurlyeq \frac{U}{1-rt}.
$$
In particular, one has $\curv_R(k)=r$.
\item
If $\q$ is a reduction of $\m$, then $C$ is described in terms of multiplicities:
$$
C=\frac{r\cdot\e(N)-\e(N^*)}{(r^2-1)\nu_R(M)}.
$$
\item
Assume $d=0$ and $R$ is positively graded over a field, and $N$ is graded.
Then
$$
C=\frac{\ell_R(N)}{(r+1)\nu_R(M)}.
$$
\end{enumerate}
\end{cor}

\begin{proof}
(1) {Set} $c=\cx_R(M)$.
Then $\cx_R(M/\q M)=c$ by \cite[Proposition 4.2.4(5)]{Av}, and there exist a real number $u>0$ and an integer $e>0$ such that $\beta^R_i(M/\q M)\le u\, i^{c-1}$ for all integers $i\ge e$.
Replacing $u$ with $\max_{1\le i\le e-1}\{u,\,i\cdot\beta_i^R(M/\q M)\}$, we may assume $e=1$.
Theorem \ref{Main1}(1) shows $P^R_k(t)=\frac{P_M^R(t)(1+t)^d}{\nu_R(M)(1-rt)}=\frac{P^R_{M/\q M}(t)}{\nu_R(M)(1-rt)}$, which implies
$$
\frac{\beta^R_i(k)}{r^i}\le\frac{1}{\nu_R(M)}\left(\nu_R(M)+u\sum_{h=1}^i\frac{h^{c-1}}{r^h}\right)
$$
for all $i\ge1$.
Since $c>0$ and $r>1$, by d'Alembert's ratio test, the series
$\sum_{h=1}^\infty\frac{h^{c-1}}{r^h}$ converges {to $T$}.
Setting $U=1+\frac{u\, T}{\nu_R(M)}$, we get $\beta^R_i(k) \le Ur^i$
for all $i\ge0$.
Therefore $P^R_k(t)\preccurlyeq \frac{U}{1-rt}$, and
$$
\curv_R(k)=\limsup_{i\to\infty}\sqrt[i]{\beta_i^R(k)}\le\left(\limsup_{i\to\infty}\sqrt[i]{U}\right)\cdot r=\left(\lim_{i\to\infty}\sqrt[i]{U}\right)\cdot r=1\cdot r=r.
$$
On the other hand, since {$\frac{P_M^R(t)}{\nu_R(M)}=\frac{P_M^R(t)}{\beta^R_0(M)}=1+(\frac{\beta^R_1(M)}{\beta^R_0(M)})t + \cdots \succcurlyeq 1$ and $(1+t)^d\succcurlyeq 1$, we have} $\frac{P_M^R(t)(1+t)^d}{\nu_R(M)}\succcurlyeq 1$. {Therefore} one has $P^R_k(t)=\frac{P_M^R(t)(1+t)^d}{\nu_R(M)(1-rt)}\succcurlyeq \frac{1}{1-rt}$.  {Hence} $\beta^R_n(k)\ge r^n$ {and so} $\curv_R(k) \geq r$.

(2) It follows from \cite[Theorems 4.7.6 and 4.7.10]{BH} that $\e(N)=\ell_R(N/\q N)$ and $\e(N^*)=\ell_R(N^*/\q N^*)$.
Thus the assertion follows from Theorem \ref{Main1}(2).

(3) Since $\ell_R(N)=\ell_R(N^*)$ by \cite[Lemma 6.6]{extd}, Theorem \ref{Main1}(2) implies the assertion.
\end{proof}

\begin{rem}\label{Sean}
A similar argument to the proof of Theorem \ref{Main1}(1) actually shows a more general statement:

Let $R$ be a $d$-dimensional Cohen-Macaulay local ring.
Let $C$ be a semidualizing $R$-module of type $r{\ne1}$.
{Let $M$ be a non-zero totally $C$-reflexive $R$-module with $\m^2M\subseteq\q M$ for some parameter ideal $\q$ of $R$.}
Then
$$
I_R^C(t)=\frac{t^d(r-t)}{1-rt},{\text{ and }P_k^R(t)=\frac{P_M^R(t)(1+t)^d}{\nu_R(M)(1-rt)}}.
$$
\end{rem}

The following result gives a sufficient condition for the assumption $\m^2M \subseteq \q M$ used in Theorem \ref{Main1}, which no longer involves a parameter ideal $\q$.

\begin{prop}\label{1854}
Let $(R,\m,k)$ be a $d$-dimensional Cohen-Macaulay local ring with $k$ infinite.
Let $M$ be a maximal Cohen-Macaulay $R$-module satisfying
\begin{equation}\label{mme}
\e(M)=\nu_R(\m M)-(d-1)\nu_R(M),
\end{equation}
Then there exists a parameter ideal $\q$ of $R$ with $\m^2M \subseteq \q M$.
\end{prop}

\begin{proof}
Since $k$ is an infinite field, one can choose a reduction $\q$ of $\m$ which is a parameter ideal of $R$; see \cite[Corollary 4.6.10]{BH}.
Consider the chain
$$
\q\m M \subseteq
\left\{
\begin{matrix}
\m^2 M\\
\q M
\end{matrix}
\right\}
\subseteq \m M \subseteq M
$$
of submodules of $M$.
We claim that $\ell_R(\q M/\q\m M)=d\cdot\nu(M)$.
In fact, as $\q$ is generated by a regular sequence on $R$ and $M$, its Koszul complex on $M$ makes an exact sequence
$$
M^{\binom{d}{2}} \xrightarrow{A} M^d \to M \to M/\q M \to 0,
$$
where $A$ is a $d\times \binom{d}{2}$ matrix with entries in $\q$.
Tensoring with $k$ the induced exact sequence $M^{\binom{d}{2}} \xrightarrow{A} M^d \to \q M \to0$ shows the claim.
Since $\ell_R(M/\q M)=\e(M)$ by \cite[Lemma 4.6.5]{BH}, the above chain of inclusions gives rise to:
$$
\ell_R(\m^2M/\q\m M)=\e(M)+(d-1)\nu_R(M)-\nu_R(\m M).
$$
Equality \eqref{mme} is nothing but the right-hand side is zero, which implies $\m^2M=\q\m M\subseteq\q M$.
\end{proof}

Recall that a Cohen-Macaulay local ring $R$ is said to have {\em minimal multiplicity} if \eqref{mme} holds for $M=R$; see \cite[Exercise 4.6.14]{BH}.
The proof of Proposition \ref{1854} shows that {if $k$ is infinite then} for a maximal Cohen-Macaulay $R$-module $M$ one has $\e(M)\geq \nu_R(\m M)-(d-1)\nu_R(M)$.
{In fact, the last inequality holds even if $k$ is a finite field. To see that, one can use the trick of passing to the flat extension $R[w]_{\m[w]}$ where $w$ is an indeterminate; see \cite[page 114, Remark]{M} for more details.}
Therefore, it makes sense to give the following definition.
\begin{dfn}\label{minMult} Let $(R,\m)$ be a Cohen-Macaulay local ring of dimension $d$. We say that a maximal Cohen-Macaulay $R$-module $M$ has \emph{minimal multiplicity} if $\e(M)= \nu_R(\m M)-(d-1)\nu_R(M)$.
\end{dfn}
Theorem \ref{Main1} and Proposition \ref{1854} immediately recover the following result given in \cite[Corollary 2.5]{Y} (see also \cite[Examples 3.5(2)]{AM}).

\begin{cor}\label{Greg}
Let $(R,\m,k)$ be a Cohen-Macaulay non-Gorenstein local ring with minimal multiplicity.
Then $R$ is G-regular.
\end{cor}
\begin{proof} Replacing $R$ with the faithfully flat extension $R[w]_{\m[w]}$ with $w$ an indeterminate, we may assume that $k$ is infinite. Then by Proposition \ref{1854} there is a parameter ideal $\q$ such that $\m^2\subseteq \q$. Therefore if $M$ is any non-zero totally reflexive module, then one has $\m^2M\subseteq \q M$. Now, Theorem \ref{Main1}({1}) implies that $\frac{P^R_M(t)}{\nu_R(M)}=\frac{P^R_R(t)}{\nu_R(R)}=1$, i.e. $M$ is a free $R$-module.
\end{proof}

In \cite[Question 6.6]{T} Takahashi asked if a local ring is Gorenstein when a syzygy of the residue field {(with respect to a minimal free resolution)} has a non-zero direct summand of finite G-dimension. As another application of Theorem \ref{Main1} we prove the following result {which specifically says that \cite[Question 6.6]{T} has affirmative answer over Artinian short local rings}.

\begin{cor}
Let $(R,\m,k)$ be a Cohen-Macaulay non-Gorenstein local ring.
Suppose that there exists a non-zero totally reflexive $R$-module $M$ with minimal multiplicity.
If $\cx_R(M)<\infty$, then for each $n\geq0$, {the $n$th syzygy} $\Omega^nk$ {of the $R$-module $k$} does not have a non-zero direct summand of finite G-dimension.
\end{cor}

\begin{proof}
 {Passing to} the faithfully flat extension $R[w]_{\m [w]}$ with $w$ an indeterminate, we may assume that $k$ is infinite.
Proposition \ref{1854} implies $\m^2M\subseteq\q M$ for some parameter ideal $\q$ of $R$.
Assume that there is an integer $n\geq 0$ such that $\Omega^nk\cong X\oplus Y$, where $X$ is a non-zero $R$-module with $\Gdim_RX<\infty$.
Replacing $n$ with a bigger integer if necessary, we may assume that $X$ is totally reflexive.

{Note that as} the map $\Hom_R(X,k)\rightarrow \Hom_R(\Omega^nk,k)$ induced from the projection $\Omega^nk\to X$ is non-zero, by the proof of \cite[Theorem 8]{Av2} there exists a power series $W(t)=\frac{(1+t)^{\alpha_1}\cdots (1+t^{2n-1})^{\alpha_n}}{(1-t^2)^{\beta_1}\cdots (1-t^{2n})^{\beta_n}}$ such that $W(t)P^R_{X}(t)\succcurlyeq P^R_k(t)$, where  $\alpha_1,\cdots,\alpha_n$, and $\beta_1,\cdots,\beta_n$ are positive integers. {Let $r$ be the type of $R$.} By Theorem \ref{Main1}(2) we have $\frac{P^R_{X}(t)}{1-rt} \preccurlyeq \frac{C\cdot P^R_M(t)}{1-rt}$ with $C$ constant. Note that multiplication by a power series with positive coefficients preserves an inequality of formal power series. Therefore by multiplying $W(t)\succcurlyeq 1$, we get
\begin{equation}\label{12345}
\frac{P^R_k(t)}{1-rt}\preccurlyeq\frac{P^R_{X}(t)W(t)}{1-rt} \preccurlyeq \frac{C\cdot P^R_M(t)W(t)}{1-rt}.
\end{equation}
Since $\frac{P^R_M(t)(1+t)^d}{\nu_R(M)}\succcurlyeq 1$, we have $\frac{P^R_k(t)}{1-rt}=\frac{P^R_M(t)(1+t)^d}{\nu_R(M)(1-rt)^2} \succcurlyeq \frac{1}{(1-rt)^2}$.

Let $\cx_R(M)=e < \infty$. Then there exists $E>\beta^R_0(M)$ such that $\beta_i^R(M)\le E\cdot i^{e-1}$ for all $i\ge 1$. Also, note that $W(t)$ has finite complexity, i.e. if $a_i$ stands for $i$th coefficient of $W(t)$, there exist $b>0$, and $B>1$ such that $a_i<B\cdot i^{b-1}$ for all $i\ge 1$.
Set $V(t)=P^R_M(t)W(t)$. Hence if $c_i$ is the $i$th coefficient of $V(t)$, then $c_i\leq BE \cdot i^{b+e-1}$ for all $i\ge 1$.
Therefore we get
$$
 \sum_{i=0}^l\frac{c_i}{r^i}\le \left(c_0+BE\sum_{i=1}^l\frac{i^{b+e-1}}{r^i}\right)\le c_0+BET,
$$
where $T=\sum_{i=0}^\infty\frac{i^{b+e-1}}{r^i}$ is a convergent power series as $r\geq 2$. Now by setting $D=C(c_0+BET)$ one has $\frac{C\cdot P^R_M(t)W(t)}{1-rt} \preccurlyeq \frac{D}{1-rt}$. Putting these together with (\ref{12345}), we get an inequality of formal power series $\frac{1}{(1-rt)^2}\preccurlyeq \frac{D}{1-rt}$, which implies that $i\cdot r^i \le D\cdot r^i$ for all $i\ge 1$, contradiction.
\end{proof}

\begin{lem}\label{L22}
Let $(R,\m,k)$ be an  Artinian non-Gorenstein local ring of type $r$.
Let $I$ be a proper ideal of $R$ containing $\m^2$ such that ${\Gdim_R}R/I{<\infty}$.
Then {$I$ is a quasi-Gorenstein ideal, and} the following hold.
\par
{\rm(1)} $\nu_R(\m)=\nu_R(I) + r$,\qquad
{\rm(2)} $\nu_R(\m/I)=\ell_R(\m/I)=r$,\qquad
{\rm(3)} $\m^2=I\m$.
\end{lem}

\begin{proof} {First we show that $I$ is a quasi-Gorenstein ideal. Since $R/I$ is totally reflexive, we only need to show that $(R/I)^*\cong R/I$; see Definition \ref{df1}. As $\m^2 (R/I)^*=0$, by Lemma \ref{type}(1) there is an exact sequence $0 \to k^{br} \to (R/I)^* \to k^b \to 0 $ with $b=\nu_R(R/I)^*$.  Since $\ell_R(R/I)=\ell_R((R/I)^*)$ by \cite[Proposition 5.2]{Ger}, we get $br+b=\ell_R((R/I)^*)=\ell_R(R/I)=r+1$,
whence $b=1$.}

{(1)} By Theorem \ref{Main1}(1) we have $P^R_k(t)=\frac{P^R_{R/I}(t)}{1-rt}$, which shows $\nu_R(\m)=\nu_R(I) + r$.

{(2)}  One has $\ell_R(R/I)=r+1$ which implies $\ell_R(\m/I)=r$. Since $\m/I$ is annihilated by $\m$, we get $\nu_R(\m/I)=\ell_R(\m/I)$.

(3) Consider the chain $I\m \subseteq\m^2 \subseteq I \subseteq \m$ of ideals of $R$.
Computing the lengths of subquotients, we have $\ell_R(\m^2/I\m) + \nu_R(\m) = \nu_R(I) + \ell_R(\m/I)$.
It follows from (1) and (2) that $\ell_R(\m^2/I\m)=0$, and we obtain $\m^2=I\m$.
\end{proof}

\begin{lem}\label{m2}
Let $(R,\m,k)$ be a $d$-dimensional Cohen-Macaulay non-Gorenstein local ring with $\m^2\ne0$ (e.g. ${d}>0$).
Then $R/\m^2$ has infinite G-dimension as an $R$-module.
\end{lem}
\begin{proof}
Assume that $R/\m^2$ has finite G-dimension.
Choose a parameter ideal $\q$ of $R$ contained in $\m^2$.
Setting $\overline R=R/\q$ and $\overline\m=\m/\q$, we have
$$
\Gdim_{\overline R}{(}\overline R/{\overline\m}^2{)}=\Gdim_R{(}R/\m^2{)}-d=0
$$
by \cite[(1.5.4) and (1.4.8)]{C}.
Using Lemma \ref{L22}(3), we get ${\overline\m}^2={\overline\m}^2\overline\m$.
Nakayama's lemma implies ${\overline\m}^2=0$, whence $\m^2=\q$.
Since $\q$ is generated by an $R$-sequence, $\m^2$ is a nonzero $R$-module of finite projective dimension.
It follows from \cite[Theorem 1.1]{LV} that $R$ is regular, which contradicts the assumption that $R$ is non-Gorenstein.
\end{proof}

{Now, we are in a position to prove Theorem \ref{C21}.}

\begin{proof}[Proof of Theorem \ref{C21}]
We use induction on $d=\dim R$. First,  {assume} $d = 0$.
Lemma \ref{L22} says $\nu_R(\m/I)=r$, where $r$ is the type of $R$. As $(\m/I)^2=0$, {$R/I$ is a Cohen-Macaulay ring with minimal multiplicity, and} one has $P^{R/I}_k(t)=\frac{1}{1-rt}${; see \cite[Example 5.2.8]{AV}}.
Theorem \ref{Main1}(1) implies $P^R_k(t)=\frac{P^R_{R/I}(t)}{1-rt}=P^R_{R/I}(t)\,P^{R/I}_k(t)$.
Now by \cite[Theorem 1.1]{Lev} $R\rightarrow R/I$ is a large homomorphism.

Assume $d \geq 1$.
Then $I$ strictly contains $\m^2$ by Lemma \ref{m2}, and {by the Prime Avoidance Theorem,} there exists a regular element $x \in I \setminus \m^2$.
Set $\overline{R}=R/(x)$ and $\overline I=I/(x)$.
Then $\Gdim_{\overline{R}}(R/I)<\infty$.
Consider the sequence $R\rightarrow \overline{R} \rightarrow R/I$ of surjective local homomorphisms.
By the induction hypothesis $\overline{R} \rightarrow R/I$ is large. Since $R\rightarrow \overline{R}$ is also large by \cite[Theorem 2.2]{Lev}, it follows that so is the composition $R\rightarrow R/I$. This proves (1).

For (2), since $\m^2 \subseteq I$, $I$ is homogeneous and $R/I$ is a graded Koszul ring. Similarly as above, the graded Poincar\'{e} series of $k$ over $R/I$ is $P^{R/I}_k(s,t)=\frac{1}{1-r(st)}$. As $R\rightarrow R/I$ is large, {we have} $$P^R_k(s,t)=P^R_{R/I}(s,t)\,P^{R/I}_k(s,t)=\frac{P^R_{R/I}(s,t)}{1-r(st)}.$$
Thus the minimal graded free resolution of $k$ over $R$ is linear if and only if so does that of $R/I$ over $R$.
\end{proof}

\begin{cor}\label{T22}
Let $(R,\m,k)$ be a $d$-dimensional Cohen-Macaulay non-Gorenstein local ring of type $r$.
Let $I$ be a quasi-complete intersection ideal of $R$ containing $\m^2$.
Then {$$ P^R_k(t)=\frac{(1-t^2)^d}{(1-t)^{{\nu_R(\m)-r}}(1-rt)}.$$}
\end{cor}
\begin{proof}
By Theorem \ref{C21}, the homomorphism $R\rightarrow R/I$ is large. Therefore, as we saw before, $P^R_k(t)=P^R_{R/I}(t)\cdot P^{R/I}_k(t)=\frac{P^R_{R/I}(t)}{(1-rt)}$. It follows from \cite[Theorem 6.2]{AHS} that $P^R_{R/I}(t)=\frac{(1-t^2)^d}{(1-t)^{\nu_R(I)}}$. It remains to show that $\nu_R(I)=\nu_R(\m)-r$. Lemma \ref{L22}(1) establishes the case $d=0$, and the case $d>0$ follows by induction.
\end{proof}

 {The following is an example for Theorem \ref{C21} and Corollary \ref{T22}.}

\begin{ex}\label{qci} Let  $R=k[{w},x,y,z]/({w^2},x^2,y^2,z^2,z{w})$. Then $R$ is an Artinian local ring with the maximal ideal $\m=({w},x,y,z)$.  {One checks $(xyz,xyw)\subseteq (0:_R\m)$ and hence $R$ is not Gorenstein.}  {Note that $R\cong P[w,z]/(w^2,z^2,wz)$, where $P=k[x,y]/(x^2,y^2)$ is a Koszul complete intersection; see \cite[4.5]{HS}. Since $P\rightarrow R$ is flat, the}  ideal $I=(x,y){R}$ is a quasi-complete intersection ideal {of $R$} containing $\m^2${; see \cite[Lemma 2.1]{AHS}}. By Lemma \ref{L22}(1), the type of $R$ is $\nu_R(\m)-\nu_R(I)=2$.
Therefore by Corollary \ref{T22} one has $P^R_k(t)=\frac{1}{(1-t)^2(1-2t)}$. Since  {$k$ has a linear resolution over $P$,} $R/I$ has linear free resolution {over $R$ and so} $R/I$ is a Koszul module. Hence $R$ is a Koszul ring by Theorem \ref{C21}(2).
\end{ex}

{The following example shows that Theorem \ref{C21} fails if $\Gdim_R(R/I)=\infty$.}

\begin{ex}\label{short} \cite[Theorem 3.5]{AHS}
Let $k$ be a field of characteristic different from $2$, and let
$$
R=\frac{k[w,x,y,z]}{(w^2,wx-y^2,wy-xz,wz,x^2+yz,xy,z^2)}.
$$
Then $R$ is an Artinian local ring with the maximal ideal $m=(w,x,y,z)$. Since $(y^2, xz) \subseteq(0:_R\m)$, $R$ is not Gorenstein. By \cite[Theorem 3.5]{AHS} $x,y$ is an exact pair of zero-divisors, and therefore $(x)$ is a quasi-complete intersection ideal of $R$.
One can easily check that $\m^2 \subseteq (x)$. Thus, Lemma \ref{L22}(1) says that type of $R$ is $3$. Hence by Corollary \ref{T22}, one has $P^R_k(t)=\frac{1}{(1-t)(1-3t)}$. {Let $I=(x-z)$. Then by using Macaulay2 one checks $\m^2\subseteq I$, and $(0:_RI)=(x-y+z,yw)$. Hence $I$ is not a quasi-Gorenstein ideal and therefore $\Gdim_R(R/I)=\infty$ by Lemma \ref{L22}. We have $\beta_2(R/I)=\nu_R(0:_RI)=2$.  If $R\rightarrow R/I$ is large then by \cite[Theorem 1.1]{Lev} we have $P^R_k(t)=P^R_{R/I}(t)\cdot P^{R/I}_k(t)$, and so $\frac{1}{(1-t)(1-3t)}=P^R_{R/I}(t)\cdot\frac{1}{1-3t}$. Hence one gets $P^R_{R/I}(t)=\frac{1}{1-t}$. Since $\beta^R_2(R/I)=2$, this is impossible.}
\end{ex}

\section{Construction}
In {this section}, we generalize results of {\cite{CGT,RSW}}. Both of {these references} use exact zero-divisors to construct infinitely many totally reflexive modules. Our ingredients for construction of such modules are quasi-Gorenstein ideals, more general  {than} exact zero-divisors.
For example, the ideal $I$ in Example \ref{qci} is a quasi-Gorenstein ideal, but since $I$ is not principal, $I$ is not generated by an exact zero-divisor; see also Example \ref{EX}.

\begin{lem}\label{lem12} Let $I$ be an ideal of $R$. If  $(0:_RI^2)=(0:_RI)$ then for every finitely generated $R/I$-module $N$, $\Hom_R(N,R/(0:_RI))=0$.
\end{lem}
\begin{proof} {Applying $\Hom_R(-,R/(0:_RI))$ to a} surjection $(R/I)^n {\twoheadrightarrow} N$,  {we get} an injection $\Hom_R(N,R/(0:_RI)) {\hookrightarrow} \Hom_R((R/I)^n,R/(0:_RI))$.
Since $\Hom_R(R/I,R/(0:_RI))\cong(0:_RI^2)/(0:_RI)=0$, we are done.
\end{proof}

\begin{lem}\label{41723}
Let $I$ be an ideal of $R$. Assume one of the following hold:
\begin{enumerate}[\rm(a)]
\item
 {$(0:_RI^2)=(0:_RI)$}, or
\item
$I$ is {quasi-Gorenstein} and $R/I$ is a G-regular ring.
\end{enumerate}
Then every exact sequence $$0 \to (R/I)^n \to M \to R/(0:_RI) \to 0$$ with $n\ge0$, as an $R$-complex, has a direct summand isomorphic to an exact sequence $0 \to (R/I)^t \to N \to R/(0:_RI) \to 0$ for some $t$ with $0\le t\le n$ such that $N$ is indecomposable.
In particular, $M\cong(R/I)^{n-t}\oplus N$.
\end{lem}

\begin{proof}
Set $J=(0:_RI)$. The assertion clearly holds if $J=0$, so let us assume $J\ne0$. Hence $\gr I=0$. Assume $M$ is decomposable and write $M\cong X\oplus Y$ for some nonzero $R$-modules $X,Y$.
Then we have an exact sequence $0 \to (R/I)^n \to X\oplus Y \xrightarrow{(f,g)} R/J \to 0$ with $n\geq 1$. {Since $R/J$ is a cyclic $R$-module and $(f,g)$ is surjective, one of $f$ or $g$ has to be surjective.}
Assume that $f$ is surjective. Then the following commutative diagram with exact rows and columns

$$
\begin{CD}
@. 0 @. 0 @. 0 \\
@. @VVV @VVV @VVV \\
0 @>>> Z @>>> X @>f>> R/J @>>> 0 \\
@. @VVV @V\binom{1}{0}VV @V{=}VV \\
0 @>>> (R/I)^n @>{(^l_h)}>> X\oplus Y @>(f,g)>> R/J @>>> 0 \\
@. @V h VV @VVV @VVV \\
0 @>>> Y @>=>> Y @>>> 0 @>>> 0 \\
@. @VVV @VVV @VVV \\
@. 0 @. 0 @. 0
\end{CD}
$$
shows that $Y$ is an $R/I$-module.
Set $Z=\Ker f$.

{The assumption (a) and Lemma \ref{lem12} imply that $\Hom_R(Y,R/J)=0$. Hence $g=0$, and the map $h:(R/I)^n\rightarrow Y$ yields a splitting $Y\rightarrow (R/I)^n$. Thus $Y\oplus Z \cong (R/I)^n$ and hence $Y$ and $Z$ are free $R/I$-modules.

Assume now that (b) applies. Then {both} $R/I$ {and $R/J$ are} totally-reflexive by Remark \ref{remqg}(2). {Hence $M$ is a totally reflexive module. As $X$ and $Y$ are direct summands of $M$, they are also totally reflexive modules.} Then Lemma \ref{lem11} shows that $Y$ is a {totally reflexive $R/I$-module, and since $R/I$ is G-regular, it is a} free $R/I$-module.

Thus in both cases, we have $Y\oplus Z \cong (R/I)^n$ and $Y\cong(R/I)^r$ for some $r$ with $0<r\le n$.}
 {Therefore there is} an exact sequence $0 \to (R/I)^{n-r} \to X \to R/J \to 0$, where $Z\cong(R/I)^{n-r}$.
Iterating this procedure, we get the desired exact sequence.
\end{proof}

\begin{lem}\label{2126}
Let $I=(x_1,\dots x_r)$ {be an ideal}, {and let} $y$ and $a_1,\dots,a_n$ be {in $R$.} Then for an $R$-module $M$ the following are equivalent.
\begin{enumerate}[\rm(1)]
\item
There exists an exact sequence $0 \to (R/I)^{n} \to M \to R/(y) \to 0$ of $R$-modules.
\item
There exists an exact sequence $R^{nr+1} \xrightarrow{T(\underline{x},y,\underline{a})} R^{n+1} \to M \to 0$, where $a_1,\dots,a_n \in (I:_R(0:_Ry))$ and   $$
T(\underline{x},y,\underline{a})=\left(\begin{array}{ccccccc|c}
x_1 & \dots & x_r & & & 0 & & a_{1} \\
& & & \ddots  & & & & \vdots \\
& 0 &  & & x_1 & \dots & x_r & a_{n} \\
\hline
& & & 0 & & & & y
\end{array}\right).
$$
\end{enumerate}
\end{lem}
\begin{proof}
See \cite[Lemmas 3.2 and 3.3]{CGT}, whose proofs are explained just before them.
\end{proof}

\begin{dfn} Let $I$ be an ideal of $R$. Define
$$
\s(I):= \sup \{ n\ge0\mid \nu_{R/I}(\mathfrak{a})=n\text{ for some ideal $\mathfrak{a}$ of $R/I$}\}.
$$
\end{dfn}

\begin{lem}
Let $I$ be an ideal of $R$ and let $y \in \m$ be an element.
For any integer $n>\s(I+(y))$ and any exact sequence $0 \to (R/I)^n \to M \to R/(y) \to 0$ of $R$-modules, $M$ has a direct summand isomorphic to $R/I$.
\end{lem}

\begin{proof}
The assertion follows from a similar argument to the proof of \cite[Proposition 3.5]{CGT} using Lemma \ref{2126}.
\end{proof}

\begin{prop}\label{2202}
Let $I$ be  {a quasi-Gorenstein} ideal of $R$ of grade  {zero}.
Assume either of the following holds.
\begin{enumerate}[\quad\rm(a)]
\item
 {$(0:_RI^2)=(0:_RI)$, or}
\item
 $R/I$ is G-regular.
\end{enumerate}  Then {for each $0<n<\s(I+(0:_RI))$, there} exists a family $$\{0 \to (R/I)^n \to M_{{n,}u} \to R/(0:_RI) \to 0\}_{u\in R}$$ of exact sequences of $R$-modules such that each  $M_{{n,}u}$ is non-free indecomposable totally reflexive $R$-module, and that if $M_{{n,}u}\cong M_{{m,}v}$ for $u,v\in R$, then {$n=m$ and} $\overline u=\overline v$ in $k$.
\end{prop}

\begin{proof}
The proof is similar to that of \cite[Theorems 3.6 and 3.8]{CGT}, {but we bring the proof for convenience of the reader.}
{By Remark \ref{remqg}(2) $(0:_RI)$ is a principal ideal. Write $(0:_RI)=(y)$, for some $y\in R$, and set $\overline{R}=R/(I+(y))$. As $n<\s(I+(y))$, we have $\nu_{\overline{R}}(\overline{a_1},\dots,\overline{a_n})=n$ for some $a_1,\dots,a_n\in R$.
Let $M$ be the cokernel of the $R$-linear map defined by $T(\underline{x},y,\underline{a})$, where $I=(x_1,\dots,x_r)$.
Since $(I:_R(0:_Ry))=(I:_RI)=R$ {by Remark \ref{remqg}(2)}, Lemma \ref{2126} provides an exact sequence
\begin{equation}\label{2136}
0 \to (R/I)^n \to M \to R/(y) \to 0.
\end{equation}
Suppose the $R$-module $M$ is decomposable. It follows from Lemma \ref{41723} that there is an exact sequence $0 \to (R/I)^{n-1} \to N \to R/(y) \to 0$, which is isomorphic to a direct summand of \eqref{2136} as an $R$-complex.
Now, Lemma \ref{2126} shows that there is an exact sequence $R^{(n-1)r+1} \xrightarrow{T(\underline{x},y,b_1,\dots,b_{n-1})} R^n \to N \to 0$.
Since $M\cong R/I\oplus N$, it is seen that $M$ has two presentation matrices $T(I,y,a_1,\dots,a_n)$ and $T(I,y,0,b_1,\dots,b_{n-1})$, both of which are $(nr+1)\times(n+1)$ matrices. Taking $nr$-th Fitting invariant, we have equality $I + (y,a_1,\dots,a_n)=I +(y,b_1,\dots,b_{n-1})$ of ideals of $R$, see \cite[Page 21]{BH}. Taking the image in $\overline{R}$, we see that $\nu_{\overline{R}}(\overline{a_1},\dots,\overline{a_n})\le n-1$, which is a contradiction.

Now choose $a_1,\dots,a_n,b\in R$ such that $\nu_{\overline{R}}(\overline{a_1},\dots,\overline{a_n},\overline{b})=n+1$.
For each $u\in R$, let $M_{{n,}u} = \Coker(T(\underline{x},y,a_1+ub,a_2,\dots,a_n))$.
Note that $\nu_{\overline{R}}(\overline{a_1+ub},\overline{a_2},\dots,\overline{a_n})=n$.
By the last argument, $M_{{n,}u}$ is an indecomposable $R$-module admitting an exact sequence of $R$-modules of the form $0 \to (R/I)^n \to M_{{n,}u} \to R/(y) \to 0$.
Let $u,v\in R$ be elements {such that} $M_{{n,}u}\cong M_{{m,}v}$ {for some $n$ and $m$}. {Clearly, one has $m=n$ by Lemma \ref{2126}.} Then, taking the $nr$-th Fitting invariants of $M_{{n,}u}$ and $M_{{n,}v}$, we see that $I+(y,a_1+ub,a_2,\dots,a_n)=I+(y,a_1+vb,a_2,\dots,a_n)$.
This induces an equality $(\overline{a_1+ub},\overline{a_2},\dots,\overline{a_n})=(\overline{a_1+vb},\overline{a_2},\dots,\overline{a_n})$ of ideals of $\overline{R}$. Now by \cite[Lemma 3.7]{CGT} we get $\overline u=\overline v$ in $k$.}
\end{proof}

 {The following example shows that a non-Gorenstein local ring that admits a totally reflexive module with minimal multiplicity, may also admit infinitely many totally reflexive modules with non-minimal multiplicities; see Definition \ref{minMult}.}

\begin{ex}\label{qci2} Let $(R,\m,k)$ and $I$ be same as Example \ref{qci}. Assume $k$ is an infinite field. Note that {since $\dim R=0$ and $\m^2(R/I)=0$, $R/I$ has minimal multiplicity, and so} $R/I$ is a G-regular ring by Corollary \ref{Greg}. Since $(0:_RI)=(xy)\subseteq I$, $\s(I+(0:_RI))=\s(I)=2$. For each $\lambda \in k$, let $M_\lambda$ be an $R$-module with a minimal presentation
$$
\begin{CD}
 R^3 @>{\left(\begin{array}{ccc}
x & y & z+\lambda t \\
0 & 0 & xy
\end{array}\right)} >> R^2 \rightarrow M_\lambda \rightarrow 0.
\end{CD}
$$
Then by Proposition \ref{2202}, $\{M_\lambda\}_{\lambda \in k}$ is an infinite family of non-free indecomposable and pairwise non-isomorphic totally reflexive $R$-modules. By Lemma \ref{2126} there exists an exact sequence $0\rightarrow R/I \rightarrow M_\lambda \rightarrow R/(xy) \rightarrow 0$. Since $\m^2$ is not contained in $(xy)$, the surjection $M_\lambda \twoheadrightarrow R/(xy)$ shows that $\m^2 M_\lambda \neq 0$. {Therefore $M_\lambda$ does not have minimal multiplicity.}
\end{ex}

{Now we are ready to prove Theorem \ref{2219}.}

\begin{proof}[Proof of Theorem \ref{2219}]
{First, we observe that $\dim R/(I+(0:_RI))=\dim_R I/I^2$. In fact,
\begin{align*}
\V(I+(0:I))&=\V(I)\cap\V(0:I)\\
&=\supp(R/I)\cap\supp(I)=\supp((R/I)\otimes I)=\supp(I/I^2).
\end{align*} }
Now, the theorem follows from Proposition \ref{2202},  and the fact if $\dim R/(I+(0:_RI))\geq 2$ then $\s(I+(0:_RI))=\infty$. Note that $\nu_R(M_{n,u})=n+1$ follows from Lemma \ref{2126}(2).
\end{proof}

\begin{lem} \label{L21} Let $Q$ be a local ring. Suppose $x_1,\dots, x_r$ and $y_1,\dots,y_r$ are $Q$-regular sequences and set $I=(x_1,\dots,x_r)$ and $J=(y_1,\dots,y_r)$. Choose any regular sequence $f_1,\dots,f_r \in IJ$ and set $R=Q/(f_1,\dots,f_r)$. Then the image of the ideal $I$ (respectively $J$) in $R$ is a {quasi-complete intersection}  ideal of grade zero.
\end{lem}

\begin{proof} Since $(f_1,\dots,f_r)$ and $I$ both are ideals of same grade $r$, $\gr IR=0$. Since $(f_1,\dots,f_r)$ and $I$ both are complete intersection ideals and $(f_1,\dots,f_r)\subset I$, $I/(f_1,\dots,f_r)$ is a quasi-complete intersection ideal of $R$; see \cite[8.9]{AHS}.
\end{proof}

\begin{cor}
Let $(Q,\n,k)$ be a G-regular local ring, and let $x_1,\dots,x_r \in \n \setminus \n^2$ be a regular sequence. Assume $\dim Q \geq r+2$. Then for any choice of a regular sequence $f_1,\dots,f_r \in (x_1,\dots,x_r)^2$ and for any integer $n\geq 1$, there exists a family $\{ M_{n,u} \}_{u\in k}$ of pairwise non-isomorphic indecomposable totally reflexive $R=Q/(f_1,\dots,f_r)$-modules such that $\nu_R( M_{n,u})=n+1$ and $\cx_R( M_{n,u})\leq r$.
\end{cor}

\begin{proof} Set $I=(x_1,\dots,x_r)R$. It follows from Lemma \ref{L21} {and Remark \ref{R11}} that $I$ is a quasi-Gorenstein ideal of $R$ with $\gr I=0$. Also, by {Proposition \ref{46}}, $R/I\cong Q/(x_1,\dots,x_r)$ is a $G$-regular ring. Note that since $f_1,\dots,f_r \in (x_1,\dots,x_r)^2$, one has $(0:_RI)\subset I$. Therefore ${\dim_RI/I^2=}\dim R/(I+(0:_RI)) =\dim R/I\geq 2$. Now the first part follows from Theorem \ref{2219}.

For the last part, consider the exact sequence $0 \to (R/I)^n \to M_{n,u} \to R/(0:_RI) \to 0$.  {Since $R/I \cong (0:_RI)=\Omega R/(0:_RI)$,}  $\Omega^j(R/(0:_RI))\cong \Omega^{j-1}(R/I)$, for all $j \geq1$. Hence we have $\cx_R(M_{n,u})\leq \cx_R(R/I)$. Note that $P^R_{R/I}(t)=\frac{1}{(1-t)^r}$; see \cite[Theorem 6.2]{AHS}. This implies that $\cx_R(R/I) \leq r$.
\end{proof}

We need the following lemma for Example \ref{EX}.

\begin{lem}\label{a}
Let $S=k[x_0,\dots,x_n]$ be a polynomial ring over {a} field $k$ with $n\ge3$.
Then $\cc=(x_i^2-x_{i-1}x_{i+1}\mid i=1,\dots,n-1)$ is a complete intersection ideal.
\end{lem}

\begin{proof}
Since $\cc$ is generated by $n-1$ elements, $\gr\cc=\height\cc \le n-1$.
{We have $S/\cc+(x_0)\cong k[x_1,\dots,x_n]/(x_1^2,x_2^2-x_1x_3,x_3^2-x_2x_4,\dots,x_{n-1}^2-x_{n-2}x_n)$, and easily see that in this ring $(x_i)^{2^i}=0$ for all $1\le i\le n-1$.
Hence $S/\cc+(x_0,x_n)$ is Artinian, which shows $\dim S/\cc\le2$.
Thus $\gr\cc=n-1$.}
\end{proof}

{In the following examples, first, we construct a class of local rings satisfying in the assumptions of Theorem \ref{2219}. Next, in (2) we construct a quasi-Gorenstein ideal containing $\m^2$ which is not a quasi-complete intersection ideal. As a result the base ring satisfy in the assumptions of Theorem \ref{C21}(1) but not Corollary \ref{T22}. Finally, in (3) we give an example of $3$-dimensional local ring that satisfy in the assumptions of Theorems \ref{C21} and \ref{2219}, and Corollary \ref{T22}.}

\begin{ex}\label{EX}
Let $Q=k[[x_0,\dots,x_n]]$ be a formal power series ring over a field $k$ with $n\geq 3$.
Then $\cc=(x_i^2-x_{i-1}x_{i+1}\mid i=1,\dots,n-1)$ is a complete intersection ideal of $Q$ by Lemma \ref{a}.
Let $\p$ be the ideal of $Q$ generated by the $2\times 2$ minors of the matrix $\left(\begin{smallmatrix}
x_0 & \cdots & x_{n-1} \\
x_1&\cdots & x_n
\end{smallmatrix}\right)$. Clearly $\p$ contains $\cc$, and it is well known that $\p$ is a Cohen-Macaulay prime ideal of height $n-1$; see \cite[Theorem 6.4]{E}.
Since $x_1\dots x_{n-1} \in (\cc:_Q\p) \backslash \p$ and $\p$ is prime, $\gr_Q(\p + ({\cc}:_Q\p))\geq n $.
Hence $\p$ is geometrically linked to $(\cc:_Q\p)$, and therefore the quotient $A=Q/(\p+(\cc:_Q\p))$ is a $1$-dimensional Gorenstein ring; see \cite{PS} and \cite[Propositions 1.2 and 1.3]{U}.
\begin{enumerate}[(1)]
\item
Then  $J=(x_1,\dots,x_n)A$.
Then $A/J\cong k[[x_0]]$ is a discrete valuation ring.
In particular, $J$ is a Gorenstein ideal of $A$ of grade zero.
{Hence $R=A[[y,z,w,v]]/(y^2,yz,z^2)$ is a $3$-dimensional Cohen-Macaulay non-Gorenstein local ring, $I=JR=(x_1,\dots,x_n)R$ is a quasi-Gorenstein ideal of $R$ of grade zero.
Also $R/I$ is {G-regular} and $I/I^2=(J/J^2)\otimes_AR$ has dimension at least $2$.}
Therefore Theorem \ref{2219}(b) applies for $R$ and $I$.
Thus for each integer $s\geq 1$ there exists a family $\{ M_{s,u} \}_{u\in k}$ of pairwise non-isomorphic indecomposable totally reflexive $R$-modules such that $\nu_R( M_{s,u})=s+1$.
\item
Let $n=4$, and let $\n=(x_0,\dots,x_4)A$ be the maximal ideal of $A$.
Since $A$ is Gorenstein, $\n$ is a Gorenstein ideal of $A$ of grade $1$.
Hence $R=A[[y,z]]/(y^2,yz,z^2)$ is a $1$-dimensional Cohen-Macaulay local ring of type $2$, and $I=\n R=(x_0,\dots,x_4)R$ is a quasi-Gorenstein ideal of grade $1$ containing $\m^2$, where $\m=(x_0,\dots,x_4,y,z)R$ is the maximal ideal of $R$.
We have $\nu_R(\m)=7$ and $\nu_R(I)=5$.
The second coefficient of the formal power series $\frac{1-t^2}{(1-t)^{5}(1-2t)}$ is 28 while by using Macaulay2, we get $\beta^R_2(k)=33$.
Hence the equality in Corollary \ref{T22} does not hold, which implies that $I$ is not a quasi-complete intersection ideal of $R$. This shows that one can not relax the hypothesis of $I$ being a quasi-Gorenstein ideal in Corollary \ref{T22}.
\item
Let $n=3$ and $R=(Q/\cc)[[y,z,w]]/(yz,zw,wy)=\displaystyle \frac{k[[x_0,x_1,x_2,x_3,y,z,w]]}{(x_1^2-x_0x_2,x_2^2-x_1x_3,yz,zw,wy)}.$
Then $R$ is a Cohen-Macaulay non-Gorenstein local ring of dimension $3$.
One has $\q=(0:_R\p)=(x_1,x_2)R$.
The ideal $\q$ is quasi-complete intersection; see Lemma \ref{L21}.
We have $\gr \q=0$, $\p\cap\q=0$ and $\dim R/\p+\q=2$. Hence Theorem \ref{2219}(a) applies for $R$. Therefore for each integer $s\geq 1$ there exists a family $\{ M_{s,u} \}_{u\in k}$ of pairwise non-isomorphic indecomposable totally reflexive $R$-modules such that $\nu_R( M_{s,u})=s+1$.
Moreover, the ideal $J=(x_0,x_1,x_2,x_3,y+z+w)R$ contains $\q$ and $J/\q$ is a complete-intersection ideal of $R/\q$. Hence $J$ is also a quasi-complete intersection ideal of $R$.
As $J$ contains $\m^2$, {the} type of $R$ is $\nu_R(\m)-\nu_R(J)=2$, and $P^R_k(t)=\frac{(1-t^2)^3}{(1-t)^5(1-2t)}$ by Corollary \ref{T22}.
\end{enumerate}
\end{ex}

\begin{ac}
The authors thank Olgur Celikbas for reading the article and {giving} comments.
They thank Sean Sather-Wagstaff for suggesting \ref{Sean}.
{Finally, they} are {very} grateful to the referee for {many} valuable suggestions and helpful comments.
\end{ac}


\end{document}